\documentclass[pdflatex,sn-mathphys-num]{sn-jnl}

\usepackage{graphicx}%
\usepackage{multirow}%
\usepackage{amsmath,amssymb,amsfonts,mathtools}%
\usepackage{amsthm}%
\usepackage{mathrsfs}%
\usepackage[title]{appendix}%
\usepackage{xcolor}%
\usepackage{textcomp}%
\usepackage{manyfoot}%
\usepackage{booktabs}%
\usepackage{algorithm}%
\usepackage{algorithmicx}%
\usepackage{algpseudocode}%
\usepackage{listings}%

\usepackage{enumerate}

\newcommand{\R}{\mathbb{R}}
\newcommand{\M}{\mathbb{M}}
\newcommand{\Z}{\mathbb{Z}}
\newcommand{\C}{\mathbb{C}}

\newcommand{\PP}{\mathcal{P}}
\newcommand{\fro}{\mathrm{F}}
\renewcommand{\d}{\,\mathrm{d}}

\DeclareMathOperator{\diag}{diag}
\DeclareMathOperator{\cau}{Cauchy}

\DeclareMathOperator{\e}{e}
\DeclareMathOperator{\Cat}{Cat}

\theoremstyle{definition}
\newtheorem{Def}{Definition}[section]
\newtheorem{exmp}[Def]{Example}

\newtheorem{rmk}{Remark}

\theoremstyle{plain}

\newtheorem{coro}[Def]{Corollary}
\newtheorem{teo}[Def]{Theorem}
\newtheorem{lemma}[Def]{Lemma}

\begin{document}

\title[Matrix algebras with orthogonal Cauchy-like- transforms]
{Matrix algebras with orthogonal Cauchy-like transforms and Toeplitz approximation: a negative result}


\author*[1]{\fnm{Dario} \sur{Fasino}}\email{dario.fasino@uniud.it}

\author[2]{\fnm{Davide} \sur{Furch\`i}}\email{dfurchi@uninsubria.it}

\author[2]{\fnm{Valerio} \sur{Loi}}\email{vloi@uninsubria.it}

\author[2,3]{\fnm{Stefano} \sur{Serra-Capizzano}}\email{s.serracapizzano@uninsubria.it}

\affil*[1]{\orgdiv{Department of Mathematics, Computer Science, and Physics}, \orgname{University of Udine}, \orgaddress{\street{Via delle Scienze 206}, \city{Udine}, \postcode{33100}, \country{Italy}}}

\affil[2]{\orgdiv{Department of Science and High Technology, Division of Mathematics}, \orgname{University of Insubria}, \orgaddress{\street{Via Valleggio 11}, \city{Como}, \postcode{22100}, \country{Italy}}}

\affil[3]{\orgdiv{Department of Information Technology, Division of Scientific Computing}, \orgname{Uppsala University}, \orgaddress{\street{Box 337}, \city{Uppsala}, \postcode{751 05}, \country{Sweden}}}


\abstract{We consider newly designed matrix algebras associated with Cauchy orthogonal transforms. In this work, we are interested in checking the approximation of Toeplitz matrices using algebras of matrices, where each matrix is diagonalized by a fixed Cauchy orthogonal transform, in the spirit of fast algorithm based on preconditioning. We use the Korovkin theory for checking the Toeplitz approximability. Unfortunately and somehow surprisingly, the Korovkin test is negative, in contrast with the successful test for all the Fourier-like and trigonometric algebras, widely used in the relevant literature. More in detail, for the Cauchy algebras, the Frobenius optimal approximation fails to lead to the desired clustering, as also confirmed by ad hoc numerical experiments. It remains open to understand which kind of structured matrices can be approximated by these new algebras.
}




\maketitle

\section{Introduction}\label{sec1}

A real Cauchy-like matrix $K_n$ of size $n$ is determined by $4n$ parameters via a nonlinear rational dependence (see Section \ref{sec:cauchy-like} for details). In contrast, the orthogonality condition $K_n^T K_n = I_n$ imposes $n(n+1)/2$ quadratic equations. Consequently, as $n$ grows, these equations become heavily overdetermined—making the existence of such matrices quite remarkable, see \cite[Thm.\ 4]{fasino-cauchy} and Theorem 2.1 here below. Beyond this theoretical interest, orthogonal Cauchy-like matrices arise naturally in solving secular equations and appear as eigenvector matrices of symmetric diagonal-plus-semiseparable matrices; see \cite{fasino-cauchy} for a review and \cite{BNS,book-semisep,etna-dpss} for further details on the subject.


Since these orthogonal Cauchy-like matrices are associated to fast transforms
\cite{trummer}, given an  orthogonal Cauchy-like matrix $K_n$, it is natural to consider the real matrix algebra of symmetric matrices that are simultaneously diagonalized by $K_n$:
\begin{equation}\label{def:alg}
\M(K_n)=\left\lbrace A\in \M_n({\R})\colon A=K_n \Lambda_n K_n^T,\, \Lambda_n\, {\rm real\, diagonal}\right\rbrace.
\end{equation}
In the last 15 years of the last century and, less intensively, after the year 2000, complex and real algebras of the form \eqref{def:alg} were considered, especially with the purpose of constructing efficient preconditioners in Krylov methods for the solution of large Toeplitz systems. More specifically, complex algebras ${\mathbb M}(U_n)$ having an expression as in \eqref{def:alg} have been considered in the relevant literature \cite{Koro-4-fourier,bf,KaO,Koro-block}, where $U_n$ is unitary and the transpose $^T$ is of course replaced by the
conjugate transpose $^{\ast}$. Among them, we have the (classical) Fourier matrix $F_n$ which leads to the famous circulant algebra ${\mathbb M}(F_n)$ \cite{Davis}. Several other examples that have been studied include the $\omega$-circulants with $|\omega|=1$ \cite{koro1,Koro-algebra1}, three other Fourier-like algebras \cite{Koro-4-fourier}, four cosine and four sine real matrix algebras \cite{KaO,Koro-algebra2}, and four Hartley matrix algebras \cite{Koro-4-fourier,bf,Koro-algebra1}. General references on the matter, especially in connection with the Toeplitz preconditioning proble,m are \cite{frobenius optimal first,chan-rev,Koro-algebra2}. 

Given $A_n$ and given the chosen fast transform matrix $U_n$, a classical approach is to consider its Frobenius optimal approximation in the algebra ${\mathbb M}(U_n)$, i.e., the matrix $\PP_{U_n}(A_n)$ such that
\[
\|A_n - \PP_{U_n}(A_n)\|_\fro = \min_{X\in {\mathbb M}(U_n)} \|A_n-X\|_\fro,
\]
where $\|\cdot\|_\fro$ denotes the Frobenius norm. It is known that the minimizer is unique and, in fact, given any $U_n$, the map $\PP_{U_n}(\cdot)$ is a linear positive operator (LPO) as emphasized in \cite{koro1}, where the connection with the Korovkin theory was made for the first time.

We want to use the Korovkin theory in a Toeplitz matrix setting for the present case of an algebra generated by an orthogonal Cauchy-like matrix like in \eqref{def:alg}, where the Toeplitz structure is characterized by the algebraic condition that entries are constant along the diagonals.

More precisely, the Toeplitz matrix generated by a function $f\in L^1(-\pi,\pi)$ and having size $n$ is defined as
\begin{equation}\label{Toepl-matrix}
 T_{n}(f):=\begin{bmatrix}a_{0}& a_{-1}& \cdot& \cdot& \cdots&a_{-(n-1)}\\
 a_{1}&a_{0}& a_{-1}&\cdots& \cdot& \cdots\\
 \cdot&a_{1}& a_{0}& a_{-1}&\cdot&\cdot\\
 \cdot&\cdot& a_{1}& a_{0}& a_{-1}& \cdot\\
 \cdot&\cdot& \cdot& \cdot& \cdot& \cdot\\
 a_{n-1}&\cdot&\cdots&\cdots& \cdot&a_{0} \end{bmatrix}
 \end{equation}
with $a_k$ for $k\in\mathbb Z$ being the Fourier coefficients of the generating function $f$,
\begin{equation}\label{Fcdef}
   a_{k} = \frac{1}{2\pi}
   \int_{-\pi}^{\pi}f(t)\e^{-{i}k t}\, \mathrm{d}t, \quad \ k\in\Z,
\end{equation}
and with $i$ denoting the imaginary unit. Notice that $f$ is real-valued if and only if $a_{-j}=\overline{a_j}$ for every $j\in\mathbb Z$ or, equivalently, if and only if $T_n(f)$ is Hermitian for every size $n$. In addition, $f$ is real-valued and even on $[-\pi,\pi]$ if and only if $T_n(f)$ is real symmetric for every size $n$. In particular, when $f$ is real-valued with $m$ being its essential infimum and $M$ being its essential supremum, we find that the spectrum of $T_n(f)$ is real, is contained in $(m,M)$ if $m<M$, and a lot of further asymptotic results hold strictly related to the generating function $f$; see \cite[Chapter 6]{GLT-BookI}.
Furthermore, the operator $T_n(\cdot)$ from $L^1(-\pi,\pi)$ to $\M_n({\C})$ is a LPO for every $n$ and for every fixed unitary matrix $U_n$,
see \cite[Corollary 6.2]{GLT-BookI}, and hence the same is true for the composition operator $\PP_{U_n}(T_n(\cdot))$,
which is the key of the Korovkin theorems in the approximation of Toeplitz matrices; see \cite{Koro-algebra1,koro1,Koro-block}.
Indeed, a successful and general approach for checking the clustering at zero of the matrix sequence $\{T_n(f) - \PP_{U_n}(T_n(f))\}_n$ for every $f\in
L^1(-\pi,\pi)$ is the check of the clustering at zero of $\{T_n(g) - \PP_{U_n}(T_n(g))\}_n$ for $g$ belonging to a finite set of functions, called the Korovkin test.

Our main result in the current work is to perform the Korovkin test to check either success or failure when approximating Toeplitz matrix sequences, and the answer is negative.
This fact is quite surprising given the fact that the Korovkin test was proven to be successful in the Toeplitz setting, also block and multilevel, for all the known trigonometric, Fourier-like, and Hartley algebras \cite{Koro-algebra1,koro1,Koro-block,Koro-4-fourier}. It remains open to understand which kind of structured matrices can be approximated by these new algebras.

The present manuscript is organized as follows. Section \ref{sec:cauchy-like} 
provides an overview of 
the Cauchy and Cauchy-like transforms and related matrix algebras. Section \ref{sec:main} presents the fundamentals of matrix Korovkin theory and our negative results, while Section \ref{sec:end} is devoted to conclusions and lists a few open problems.

\section{Cauchy matrices}\label{sec:cauchy-like}

In this section we present Cauchy and Cauchy-like matrices and summarize the results needed for our purposes.

\subsection{Orthogonal Cauchy-like matrices}\label{ssec:cauchy-like}
A matrix $C\in\M_n(\R)$ is of \emph{Cauchy} type if $C=(1/(x_k-y_j))_{k,j=1}^{n}$ where $x_k,y_j$ are mutually distinct real numbers, called nodes. We will write $C_n=\cau(x,y)$.

Denote $D_x=\diag(x)\in\M_n(\R)$ the diagonal matrix with diagonal entries given by a vector $x\in\R^n$. Let $x_k,y_j\in\R$ be pairwise distinct numbers, for $k,j=1,...,n$, and let $C=\cau(x,y)$. A \emph{Cauchy-like} matrix with nodes $x_k,y_j$ is a matrix of the form $K=D_vCD_w$ for $v,w\in\R^n$.
The displacement operator
\begin{equation*}
  \mathcal{D}_{x,y}\colon\M_n(\R)\to\M_n(\R),\quad X\mapsto D_xX-XD_{y}
\end{equation*}
characterizes the Cauchy-like matrix $K=D_vCD_w$ as the unique solution of the equation $\mathcal{D}_{x,y}(X)=vw^T$ for $X\in\M_n(\R)$.
The formula for the determinant of a Cauchy matrix \cite{schechter},
\begin{equation*}
  \det(\cau(x,y))=\frac{\prod_{k>j}(x_k-x_j)(y_j-y_k)}{\prod_{k,j}(x_k-y_j)},
\end{equation*}
shows that every Cauchy matrix is nonsingular. A formula for the inverse follows from the equivalence
\begin{equation*}
\mathcal{D}_{y,x}(C^{-1})=D_{y}C^{-1}-C^{-1}D_{x}=C^{-1}(CD_y-D_xC)C^{-1}=-C^{-1}\mathbf{1}\mathbf{1}^TC^{-1}=-ab^T
\end{equation*}
where $a=C^{-1}\mathbf{1}$ and $b=C^{-T}\mathbf{1}$
with $\mathbf{1} = (1,\ldots,1)^T$. Then, $C^{-1}$ is the Cauchy-like matrix
\begin{equation}\label{inv}
  C^{-1}=-D_a\cau(y,x)D_b=D_aC^TD_b.
\end{equation}
For $k=1,\ldots,n$, it is known that
\begin{align}\label{formulaab}
  a_k=-\frac{p(y_k)}{q^{\prime}(y_k)},\qquad b_k=-\frac{q(x_k)}{p^{\prime}(x_k)},
\end{align}
where
\begin{equation*}
  p(x)=\prod\limits_{k=1}^{n}(x-x_k),\qquad q(x)=\prod\limits_{k=1}^{n}(x-y_k).
\end{equation*}
Let $\mathcal{K}_n$ be the set of Cauchy matrices $C\in\M_n(\R)$ for which there exist $v,w\in\R^n$ such that $D_vCD_w$ is orthogonal.

\begin{rmk}\label{ordering}
As $\mathcal{K}_n$ is invariant under row/column permutation, we can safely assume that the vectors $x$ and $y$ are ordered in ascending order in what follows.
\end{rmk}

We state an important result from \cite{fasino-cauchy}.

\begin{teo}\label{th-main-fasino}
Let $C=\cau(x,y)$ where the vectors $x$ and $y$ are ordered increasingly. Then, $C\in\mathcal{K}_n$ if and only if the nodes interlace, that is, either $x_1<y_1<x_2<y_2<\ldots<x_n<y_n$ or $y_1<x_1<\ldots<y_n<x_n$. More precisely, we have the first sequence of inequalities when, setting $a=C^{-1}\mathbf{1}$, $b=C^{-T}\mathbf{1}$, the numbers $a_i,b_i$ are all negative and the second sequence when they are all positive.
Furthermore, the Cauchy-like matrix $K=D_vCD_w$ is orthogonal if and only if there exists a scalar $\alpha\neq 0$ such that
\begin{equation}\label{alpha}
  v_k^2=\alpha b_k,\quad w_k^2=a_k/\alpha,\quad k=1,\ldots,n.
\end{equation}
\end{teo}

Note that the Cauchy-like matrix $K$ in the above result is independent of the scalar $\alpha$,
and is orthogonal if and only if it
has the form $D_{|b|}^{1/2}CD_{|a|}^{1/2}$ where we can choose any sign of the square roots of $|b_k|$ and $|a_k|$. Thus, from any Cauchy matrix $C\in\M_n(\R)$ with interlaced nodes, we can obtain $2^{2n-1}$ orthogonal Cauchy-like matrices.

Let $\mathcal{C}_n$ and $\mathcal{T}_n$ be the sets of $n\times n$ circulant and Toeplitz matrices, respectively. It is not hard to see that $\mathcal{K}_n\cap\mathcal{C}_n=\varnothing$.
On the other hand, $\mathcal{K}_n\cap\mathcal{T}_n$ is not empty and contains the matrices that can be written as
\begin{equation*}
    \left[ \frac{r}{j-k+s} \right]_{j,k=1}^n
\end{equation*}
for some numbers $r,s$. Matrices of this form are known as Cauchy-Toeplitz, see \cite{Ty1,Ty2}.
\newline

Let $K=D_v\cau(x,y)D_w$ be an orthogonal Cauchy-like matrix of size $n$. We can write the square of its $(k,j)$-entry as
\begin{align*}
  (K_{kj})^2 =
  \left(\frac{w_jv_k}{x_k-y_j}\right)^2
  & = \frac{\prod\limits_{\substack{s=1 \\ s\neq k}}^{n}(y_j-x_s)}
  {\prod\limits_{\substack{s=1 \\ s\neq j}}^{n}(y_j-y_s)}
  \frac{\prod\limits_{\substack{s=1 \\ s\neq j}}^{n}(x_k-y_s)}
  {\prod\limits_{\substack{s=1 \\ s\neq k}}^{n}(x_k-x_s)} \\
  & = \prod\limits_{\substack{s=1 \\ s\neq k}}^{n}\frac{(y_j-x_s)}{(x_k-x_s)}
  \prod\limits_{\substack{s=1 \\ s\neq j}}^{n}\frac{(x_k-y_s)}{(y_j-y_s)}
  = L_k^{(x)}(y_j){L}_{j}^{(y)}(x_k) ,
\end{align*}
where for $k=1,\ldots,n$ the functions $L_k^{(x)}(z)$ and $L_k^{(y)}(z)$ are the basic Lagrange polynomials based on the nodes $\lbrace x_k\rbrace_{k=1}^{n}$ and $\lbrace y_k\rbrace_{k=1}^{n}$, respectively.
Consequently, from the orthogonality of $K$ we get the following interesting identity involving Lagrange polynomials with interlaced nodes.

\begin{coro}
Let the nodes $x_k,y_k$ fulfill the interlacing condition in Theorem
\ref{th-main-fasino}.
In the notation above, for $j = 1,\ldots,n$ we have
\begin{equation*}
   \sum_{k=1}^n L_k^{(x)}(y_j){L}_{j}^{(y)}(x_k)=
   \sum_{k=1}^n L_j^{(x)}(y_k){L}_{j}^{(y)}(x_j) =1.
\end{equation*}
\end{coro}

Finally, since the set $\mathcal{K}_n$ is invariant under affine transformations of the real line applied on the nodes, we can consider to take nodes in the interval $[0,1]$.


\subsection{The matrix algebra of orthogonal Cauchy-like matrices}\label{ort}

We restate the following result from \cite[Subsection 4.4]{fasino-cauchy},
with the sole difference that we allow complex matrices in the algebra.

For a given unitary matrix $X\in\M_n(\C)$ define the commutative algebra
\begin{equation*}
  \mathbb{A}(X)=\left\lbrace A\in\M_n(\C)\colon A=X\Lambda X^{H},\text{ $\Lambda$ complex diagonal}\right\rbrace.
\end{equation*}
The following corollary provides a complete description of $\mathbb{A}(X)$ when $X$ is an orthogonal Cauchy-like matrix. It is the generalization of \cite[Corollary 13]{fasino-cauchy} in our context, the proof is identical.

\begin{coro}\label{lambda}
Let $K=D_v\cau(x,y)D_w$ be orthogonal and $v,w$ normalized so that the constant $\alpha$ in \eqref{alpha} equals $1$. Then $A\in\mathbb{A}(K)$ if and only if
\begin{equation*}
  \mathcal{D}_{x,x}(A)=vz^T-zv^T
\end{equation*}
where $z=Av$. In this case, the vector of eigenvalues of $A$ is $\lambda=C^TD_vz$.
\end{coro}

The relation $\lambda=C^{T}D_vz$ establishes an isomorphism $\C^n\to\mathbb{A}(K)$,
which enables a generic matrix $A\in\mathbb{A}(K)$ to be recovered from the knowledge of the vectors $x,v$ and $z$. See \cite{fasino-cauchy} for a detailed discussion.


\section{The Korovkin Theorem}\label{sec:main}

In this section, we present a Korovkin-type theory for Toeplitz matrices 
using Cauchy-like orthogonal matrices.
More precisely, we investigate whether the commutative algebras generated by a sequence of Cauchy-like orthogonal matrices approximate a specific class of Toeplitz matrices. This is achieved by considering the approximation of Toeplitz matrices associated with suitable test functions. Our results provide a negative answer.


\subsection{The Korovkin Theorem in a Toeplitz setting}

Firstly, we present some of the theory developed in \cite{koro1} and restated in \cite{koro-tutorial}.
In the sequel, $\lambda(X)$ and $\sigma(X)$ denote the set of eigenvalues and singular values of the matrix $X\in\M_n(\C)$, respectively.

\begin{Def}
A matrix sequence $\lbrace A_n\rbrace$, $A_n\in\M_n(\C)$, is \emph{clustered at} $s\in\C$ in the eigenvalue sense (singular value sense with
$s\in\R^+_0$), if for any $\varepsilon >0$ holds
\begin{equation*}
  \#(\lambda(A_n)\cap\lbrace z\mid\arrowvert z-s\arrowvert <\varepsilon\rbrace)=o(n)\quad\bigg(\#(\sigma(A_n)\cap\lbrace z\mid\arrowvert z-s\arrowvert <\varepsilon\rbrace)=o(n)\bigg).
\end{equation*}
The cluster is \emph{strong} if the term $o(n)$ can be replaced by $O(1)$.
\end{Def}

Further, we remind that the clustering at $s\in\C$ for the matrix sequence $\lbrace A_n\rbrace$ is equivalent to the sequence being distributed as the constant function $s$ in the eigenvalue/singular value sense, as discussed in, e.g., \cite{GLT-BookI}. In the Hermitian setting, if $\lbrace A_n\rbrace$ is {clustered at} $s\in\R$ in the eigenvalue sense then $\lbrace A_n\rbrace$ is also {clustered at} $|s|$ in the singular value sense. The case $s=0$ is especially important in the preconditioning theory: the set of matrix sequences $\lbrace A_n\rbrace$ which are {clustered at} $0$ in the singular value sense form a two-sided ideal of the $*$-algebra of GLT matrix sequences \cite{GLT-BookI}, and this is very useful for simplifying the analysis of tough matrix sequences.

Given a sequence $\lbrace U_n\rbrace$ of unitary matrices we consider the sequence of LPOs
\begin{equation}\label{eq:lpo}
  \PP_{U_n}\colon\M_n(\C)\to\mathbb{A}(U_n),\quad  X\mapsto\arg\min_{A\in\mathbb{A}(U_n)}\|X-A\|_\fro .
\end{equation}
The explicit formula for this projection operator is
\begin{equation}\label{eq:explicit}
  \PP_{U_n}(X)=U_n\diag(U_n^{\ast}XU_n)U_n^{\ast},
\end{equation}
with $\diag(X)$ being the diagonal matrix having $X_{k,k}$ as $k$-th diagonal entry. Notice that the explicit expression in (\ref{eq:explicit})
yields a self-evident proof of the LPO nature of the operators in (\ref{eq:lpo}).

Pose
\begin{equation*}
  C_{2\pi}=\lbrace f\colon\R\to\C\text{ continuous and $2\pi$-periodic}\rbrace.
\end{equation*}
For $f\in C_{2\pi}$, we apply the projection to the Toeplitz matrix sequences $\lbrace T_n(f)\rbrace$ where the $n$-th matrix is
as in \eqref{Toepl-matrix} and \eqref{Fcdef}. For every $n$, the map $T_n\colon C_{2\pi}\to\M_n(\C)$ is LPO.

For $k=1,\ldots,n$, the $k$-th eigenvalue of $\PP_{U_n}(T_n(f))$ is given by
\begin{equation}\label{eigenvalue}
  \lambda_{k}(f,n)=(U_n^{\ast}T_n(f)U_n)_{kk} .
\end{equation}
We emphasize that, for a fixed value of $n$, this expression defines an LPO, $\lambda_k(\cdot,n)\colon C_{2\pi}\to\C$.
In the sequel, we will apply \cite[Theorem 5.9]{koro-tutorial}, which is restated here.

\begin{teo}\label{teo1}
Consider the Banach space $C_{2\pi}$ endowed
with the sup-norm, define $t_k^{(n)}=2\pi (k-1)/n$ for $k=1,\ldots,n$ and let us take $\lambda_k$ as in \eqref{eigenvalue}. If
\begin{equation}\label{cond}
  \max_{k=1,\ldots,n}|\lambda_k(\e^{it},n)-\e^{it_k^{(n)}}|=O(\theta_n)
\end{equation}
for some sequence  $\{\theta_n\}_n$ tending to zero as $n\to\infty$, then for every $f\in C_{2\pi}$ the sequence $\lbrace T_n(f)-\PP_{U_n}(T_n(f))\rbrace$ is clustered at zero. The cluster is strong if we can set $\theta_n=n^{-1}$.
\end{teo}

If we restrict our attention to real valued functions, the theorem becomes
as follows.

\begin{teo}\label{teo2}
In the hypothesis of Theorem~\ref{teo1}, if
\begin{equation}\label{cond2}
  \max_{k=1,\ldots,n}|\lambda_k(g(t),n)-g(t_k^{(n)})|=O(\theta_n)
\end{equation}
for $g(t)\in\lbrace\cos{t},\sin{t}\rbrace$
and some sequence  $\{\theta_n\}_n$ tending to zero as $n\to\infty$,
then for every real valued $f\in C_{2\pi}$ the sequence $\lbrace T_n(f)-\PP_{U_n}(T_n(f))\rbrace$ is clustered at zero.
The cluster is strong if we can set $\theta_n=n^{-1}$.
\end{teo}

If moreover we restrict to even functions we obtain

\begin{teo}\label{teo3}
In the hypothesis of Theorem~\ref{teo1}, if
\begin{equation}\label{cond3}
  \max_{k=1,\ldots,n}|\lambda_k(g(t),n)-g(t_k^{(n)})|=O(\theta_n)
\end{equation}
for $g(t)\in\lbrace\cos{t},\cos{2t}\rbrace$
and some sequence  $\{\theta_n\}_n$ tending to zero as $n\to\infty$,
then for every even real valued $f\in C_{2\pi}$ the sequence $\lbrace T_n(f)-\PP_{U_n}(T_n(f))\rbrace$ is clustered at zero.
The cluster is strong if we can set $\theta_n=n^{-1}$.
\end{teo}

Corollary~\ref{lambda} already suggests that the Korovkin test will fail, since the displacement rank of the Toeplitz sequence is generally greater than one.
In particular, it is easy to prove that the first two theorems cannot be applied to our setting.

\begin{coro}   \label{corosin}
No sequence of matrix algebras $\mathbb{A}(K_n)$ with $K_n$ orthogonal Cauchy-like matrix can ever satisfy the hypotheses \eqref{cond} of Theorem \ref{teo1}
and \eqref{cond2} of Theorem \ref{teo2}.
\end{coro}
\begin{proof}
Firstly, consider Theorem \ref{teo1}.
For $f\in C_{2\pi}$, $T_n(f)$ in general fails to be a real symmetric matrix while its approximant $\PP_{K_n}(T_n(f))$ is real symmetric by construction. Indeed, the $(k,j)$-entry of the Toeplitz matrix $T_n(\e^{it})$ equals $1$ if $k-j=1$ and zero otherwise. Hence, the matrix $K_n^{\ast}T_n(f)K_n=K_n^TT_n(f)K_n$ is real and so are the eigenvalues of $\PP_{K_n}(T_n(\e^{it}))$ by formula (\ref{eq:explicit}). Therefore, \eqref{cond} cannot be satisfied.

Now, consider Theorem~\ref{teo2}.
The $(k,j)$-entry of the Toeplitz matrix $T_n(\sin{t})$ equals $1/2i$ if $k-j=-1$,
$-1/2i$ if $k-j=1$ and zero otherwise. Thus $T_n(\sin{t})$ is a real skew-symmetric matrix multiplied by $i$. This property is inherited by $K_n^{\ast}T_n(\sin{t})K_n$, so its diagonal is zero.
Finally, we have $\PP_{K_n}(T_n(\sin{t})) = O$.
\end{proof}

We leave the check for Theorem~\ref{teo3} to the next subsection where a more general result is presented.

\begin{rmk}
In particular, condition \eqref{cond3} implies that for any $a<b\in\R$ it holds the limit
\begin{equation}\label{lebcond}
    \frac{\#\left\lbrace \lambda_k(\cos t,n)\in[a,b]\right\rbrace}{n}\longrightarrow\frac{m\left(\left\lbrace t\in[0,2\pi]\mid\cos(t)\in[a,b]\right\rbrace\right)}{2\pi}\qquad\text{as $n\rightarrow\infty$,}
\end{equation}
where $m$ stands for the Lebesgue measure.
\end{rmk}

\subsection{A direct approximation obstruction for Toeplitz sequences generated by real functions}


Our next goal is to extend the negative result in Corollary \ref{corosin}
to real-valued symbols.
We start by proving a technical lemma that provides a lower bound on the $\infty$-norm of the columns of an orthogonal Cauchy-like matrix.
Incidentally, this lemma offers an alternative proof of certain well-known estimates on Catalan numbers. We include these estimates to make the lemma self-contained.

For a vector $u\in\C^{n}$, we denote $\|u\|_\infty=\max_{j=1,\ldots,n}|u_j|$.
Recall the Catalan numbers
\begin{equation*}
   \Cat_0=1,\qquad\Cat_n=\frac1{n+1}\binom{2n}{n}=\sum_{j=1}^n\Cat_{j-1}\Cat_{n-j}\qquad\text{for $n\geq 1$.}
\end{equation*}

\begin{lemma}
\label{prop:geometric-localization}
Let $K=D_v\cau(x,y)D_w=\begin{bmatrix} u_1 & \cdots & u_n\end{bmatrix}\in\M_n(\R)$ be an orthogonal Cauchy-like matrix. Then
\begin{equation}\label{eq:geometric-localization}
 \left(\prod_{k=1}^n\|u_k\|_\infty\right)^{1/n}\geq\Cat_n^{-1/n}\geq\frac{1}{4}
 \qquad\text{and}\qquad\Cat_n^{-1/n}\longrightarrow\frac{1}{4}.
\end{equation}
\end{lemma}
\begin{proof}
Let $C=\cau(x,y)$. For a general matrix $B$ of dimension $r$, the device we use to prove our claim is to consider the largest product in the determinant expansion. Precisely, we define
\begin{equation}\label{eq:prod}
   M(B):= \max_{\sigma \in S_r} \prod_{j=1}^r |B_{\sigma(j),j}|.
\end{equation}
We first prove that
\begin{equation} \label{eq:cauchy-permutation-bound}
 |\det C|\leq M(C)\Cat_n.
\end{equation}
We prove \eqref{eq:cauchy-permutation-bound} by induction,
simultaneously for both interlacing orientations. The case $n=0,1$ are clear. Without loss of generality, we can assume that
\begin{equation*}
x_1 < y_1 < x_2 < \cdots < x_n < y_n.
\end{equation*}
For $1\leq j\leq n$, let $C^{(j)}$ be the minor obtained by deleting row $1$ and column $j$.

Since this deletion destroys interlacing, we further split this minor in diagonal blocks, by considering the block $L_j$ containing rows $2,\ldots,j$ and columns $1,\ldots,j-1$ of $C$, and let $R_j$ contain rows and columns $j+1,\ldots,n$. The left block preserve the interlacing, but in the opposite direction. For this reason, in the induction step we will cover both orientations.

We use now Cauchy determinant formula, obtaining.
\begin{equation*}
 |\det C^{(j)}|
 =
 |\det L_j|\,|\det R_j|
 \prod_{p=2}^{j}\prod_{q=j+1}^{n}
 \frac{(x_q-x_p)(y_q-y_{p-1})}
      {(y_q-x_p)(x_q-y_{p-1})}.
\end{equation*}
Here, every mixed factor has magnitude smaller than 1. Indeed, since $y_{p-1}<x_p<x_q<y_q$, we obtain
\begin{equation*}
 (y_q-x_p)(x_q-y_{p-1})-(x_q-x_p)(y_q-y_{p-1})=(x_p-y_{p-1})(y_q-x_q)>0.
\end{equation*}
Consequently,
\begin{equation}
 |\det C^{(j)}|
 \leq|\det L_j|\,|\det R_j|.
 \label{eq:minor-block-factorization}
\end{equation}
Now, the induction hypothesis and
\eqref{eq:minor-block-factorization} give us
\begin{equation*}
 |\det C^{(j)}|\leq M(L_j)M(R_j)\Cat_{j-1}\Cat_{n-j}.
\end{equation*}
Now, choose permutation products attaining $M(L_j)$ and $M(R_j)$.
Together, these select one entry from every row except row $1$
and every column except column $j$, without repetitions.
Adjoining $C_{1j}$ therefore gives a permutation product of $C$,
not necessarily maximizing, and hence
\begin{equation*}
 |C_{1j}|M(L_j)M(R_j)\leq M(C).
\end{equation*}
Expansion along the first row now yields
\begin{align*}
 |\det C|
 &\leq\sum_{j=1}^{n}|C_{1j}||\det C^{(j)}|\\
 &\leq\sum_{j=1}^{n}\Cat_{j-1}\Cat_{n-j}|C_{1j}|M(L_j)M(R_j)\\
 &\leq M(C)\sum_{j=1}^{n}\Cat_{j-1}\Cat_{n-j}
 =M(C)\Cat_n.
\end{align*}
This proves \eqref{eq:cauchy-permutation-bound}.
The Catalan recurrence arises from splitting $\det C^{(j)}$ 
into determinants of sizes $j-1$ and $n-j$, then summing over the deleted column $j$.

We now return to $K=D_vCD_w$. Indeed, the reason for using permutation
products is that every such product acquires the same scaling
factor
\begin{equation*}
 P:=\left(\prod_{i=1}^{n}|v_i|\right)
    \left(\prod_{j=1}^{n}|w_j|\right).
\end{equation*}
For every $\sigma\in S_n$,
\begin{equation*}
 \prod_{j=1}^{n}|K_{\sigma(j),j}|
 =
 \left(\prod_{j=1}^{n}|v_{\sigma(j)}|\right)
 \left(\prod_{j=1}^{n}|w_j|\right)
 \prod_{j=1}^{n}|C_{\sigma(j),j}|=P\prod_{j=1}^{n}|C_{\sigma(j),j}|.
\end{equation*}
As a consequence
\begin{equation*}
 M(K)=PM(C)\qquad\text{and}\qquad|\det K|=P|\det C|.
\end{equation*}
Applying \eqref{eq:cauchy-permutation-bound}, and then
bounding the selected entry in each column by that column's maximum, we obtain
 \begin{equation*}
 1=|\det K|=P|\det C|\leq P M(C)\Cat_n=M(K)\Cat_n\leq\Cat_n\prod_{j=1}^{n}\|u_j\|_\infty.
\end{equation*}
Taking $n$th roots proves
\begin{equation*}
 \left(\prod_{j=1}^{n}\|u_j\|_\infty\right)^{1/n}
 \geq\Cat_n^{-1/n}.
\end{equation*}
Finally, the central binomial coefficient is the largest
among the $2n+1$ coefficients whose sum is $4^n$. Hence
\begin{equation*}
 \frac{4^n}{(n+1)(2n+1)}
 \le
 \Cat_n
 =\frac{1}{n+1}\binom{2n}{n}
 \leq4^n.
\end{equation*}
It follows that
\begin{equation*}
 \frac14
 \leq\Cat_n^{-1/n}
 \leq\frac{\bigl((n+1)(2n+1)\bigr)^{1/n}}{4}
 \longrightarrow \frac14.
\end{equation*}
\end{proof}

The argument below constructs a general obstruction for the minimal Frobenius preconditioning based on the following idea: in our setting, a small approximation error forces the trigonometric polynomials associated to the basis given by orthogonal Cauchy-like matrices to accumulate the mass where the symbol is constant. On the other hand, we have the rigid interlacing condition that forces column maxima not to be small on average.

Let $\d m=\d t/(2\pi)$ be the normalized Lebesgue measure on $[0,2\pi]\subseteq\R$. For real-valued $f\in C_{2\pi}$, let
\begin{equation*}\label{eq:value-distribution}
 m_f(B)=m\left(\lbrace t\in[0,2\pi]\colon f(t)\in B\rbrace\right),\qquad
 Q_f(r)=\sup_{a\in\R}m_f([a-r,a+r]),\qquad r\ge0,
\end{equation*}
where $B\subseteq\R$ is Borel, and write $s_+=\max\{s,0\}$. The measure $m_f$ is thus the distribution of the values of $f$, while $Q_f(r)$ measures the largest measure of the circle on which $f$ can lie at distance $r$ from one value. These quantities control the asymptotic lower bound of the normalized Frobenius error as follows. If not stated differently, all the integrals must be intended on the interval $[0,2\pi]$.

\begin{teo}\label{teo:concentration-obstruction}
Let $\lbrace K_n\rbrace$ be a sequence of orthogonal Cauchy-like matrices and let $f\in C_{2\pi}$ be real-valued. For every $\delta>0$,
\begin{equation}\label{eq:concentration-obstruction}
 \liminf_{n\to\infty}\frac{1}{n}
 \left\|T_n(f)-\PP_{K_n}\bigl(T_n(f)\bigr)\right\|_{\fro}^2
 \ge
 \delta^2\left(\frac14-\sqrt{Q_f(\delta)}\right)_+^2.
\end{equation}
In particular, if
\begin{equation*}\label{eq:small-atoms}
 \sup_{q\in\R}m_f(\lbrace q\rbrace)<\frac1{16},
\end{equation*}
then
\begin{equation*}\label{eq:positive-obstruction}
 \liminf_{n\to\infty}\frac{1}{n}
 \left\|T_n(f)-\PP_{K_n}\bigl(T_n(f)\bigr)\right\|_{\fro}^2>0.
\end{equation*}
\end{teo}
\begin{proof}
Write $K_n=\begin{bmatrix} u_1^{(n)} & \cdots & u_n^{(n)}\end{bmatrix}\equiv\begin{bmatrix} u_1 & \cdots & u_n\end{bmatrix}$, and set
\begin{equation*}
 E_n\coloneqq\left\|T_n(f)-\PP_{K_n}(T_n(f))\right\|_{\fro}.
\end{equation*}
For every $z\in\mathbb C^n$, 
\begin{equation*}\label{eq:toeplitz-quadratic-form}
 z^HT_n(f)z=\int f(t)\left|\sum_{j=1}^nz_j\e^{i(j-1)t}\right|^2 \d m,
\end{equation*}
and $\|T_n(f)\|\leq\|f\|_\infty$. Given the polynomials
\begin{equation*}
 p_k(t)=\sum_{j=1}^n(u_k)_j \e^{i(j-1)t}
\end{equation*}
for $k=1,\ldots,n$, we recall also Parseval and orthogonality properties
\begin{equation}
 \int |p_k|^2 \d m=1,
 \qquad
 \sum_{k=1}^n|p_k(t)|^2=n\qquad\text{for every $t$.}
 \label{eq:fourier-column-identities}
\end{equation}
In particular, we have
\begin{equation*}
 \lambda_k\equiv\lambda_k(f(t),n)=u_k^HT_n(f)u_k=\int f|p_k|^2 \d m\in\R.
\end{equation*}
By using the fact that the Frobenius norm is unitarily invariant, and using a direct computation of the Frobenius norm, we have
\begin{equation}
 E_n^2=\|T_n(f)\|_{\fro}^2-\sum_{k=1}^n\lambda_k^2
 \label{eq:direct-projection-energy}
\end{equation}
Equation~\eqref{eq:direct-projection-energy} subtracts the ``squared means''
$\lambda_k^2$, thus the natural step is analyzing
\begin{equation*}
    \int |f-\lambda_k|^2|p_k|^2 \d m=\|(f-\lambda_k)p_k\|_{L^2}^2,
\end{equation*}
which represents the variance of $f$ under the probability measure $|p_k|^2\d m$. Its square root represents the continuous version of the residual $\|(T_n(f)-\lambda_k I_n)u_k\|_2$. Expanding the square, we have
\begin{align*}
    \int |f-\lambda_k|^2|p_k|^2\d m
    &=\int f^2|p_k|^2dm-2\lambda_k\int f|p_k|^2 
    \d m+\lambda_k^2\int |p_k|^2 \d m\\
    &=\int f^2|p_k|^2 \d m-\lambda_k^2 .
\end{align*}
This is just the moment formulation of the variance of $f$. The second identity in \eqref{eq:fourier-column-identities} removes the individual polynomials, and \eqref{eq:direct-projection-energy} makes the Frobenius error reappear in the following way
\begin{align*}
\sum_{k=1}^n \int |f-\lambda_k|^2|p_k|^2 \d m
 &=\sum_{k=1}^n\left(\int f^2|p_k|^2 \d m-\lambda_k^2\right)\\
 &=n\int f^2 \d m-\sum_{k=1}^n\lambda_k^2=E_n^2+n b_n
\end{align*}
Here, we have
\begin{equation*}
 b_n\coloneqq\int f^2 \d m-\frac{1}{n}\|T_n(f)\|_{\fro}^2=\sum_{\ell\in\mathbb Z}
 \min\left\lbrace\frac{|\ell|}{n},1\right\rbrace|a_\ell(f)|^2
 \ge0,\qquad b_n\longrightarrow0.
\end{equation*}
Indeed, the following key equalities hold
\begin{equation*}
    \|T_n(f)\|_{\fro}^2
 =\sum_{|\ell|<n}(n-|\ell|)|a_\ell(f)|^2\qquad\text{and}\qquad\sum_{\ell\in\mathbb{Z}}|a_\ell(f)|^2=\int f^2 \d m.
\end{equation*}
In particular, the convergence of the sequence $b_n$ then follows by dominated convergence for the summable sequence $|a_\ell(f)|^2$. Fix $\delta>0$ and set
\begin{equation*}
    S_k\coloneqq\lbrace t\colon|f(t)-\lambda_k|\leq\delta\rbrace.
\end{equation*}
By definition, $m(S_k)\leq Q_f(\delta)$. Fourier inversion, the Cauchy-Schwarz inequality and the first identity in \eqref{eq:fourier-column-identities} give, for every $j$,
\begin{align}
 |(u_k)_j|
 & = \left|\int p_k(t)\e^{-i(j-1)t} 
 \d m(t)\right| \leq
 \int_{S_k}|p_k| \d m
 +\int_{S_k^c}|p_k| \d m\notag\\
 & \leq\left(\int_{S_k} \d m\right)^{1/2}
 \left(\int_{S_k}|p_k|^2 \d m\right)^{1/2}
 +\int_{S_k^c}|p_k| \d m\notag\\
 &\leq\sqrt{m(S_k)}
   +\left(\int_{S_k^c}|p_k|^2 \d m\right)^{1/2}\notag\\
 &\leq\sqrt{Q_f(\delta)}+\frac{\left(\int |f-\lambda_k|^2|p_k|^2 \d m\right)^{1/2}}{\delta},\label{norminf}
\end{align}
where, for the last inequality, we used
\begin{equation*}
    \delta^2\int_{S_k^c}|p_k|^2 \d m
    \leq\int_{S_k^c}|f-\lambda_k|^2|p_k|^2 \d m
    \leq\int |f-\lambda_k|^2|p_k|^2 \d m.
\end{equation*}
Using Lemma~\ref{prop:geometric-localization}, the arithmetic-geometric mean inequality, the inequality on $\|u_k\|_{\infty}$ we obtain from \eqref{norminf} and the Cauchy-Schwarz inequality, we get the following chain of inequalities
\begin{align*}
 \Cat_n^{-1/n}
 & \leq\left(\prod_{k=1}^n\|u_k\|_\infty\right)^{1/n}
 \leq\frac1n\sum_{k=1}^n\|u_k\|_\infty\\
 & \leq\sqrt{Q_f(\delta)}+
 \frac1{n\delta}\sum_{k=1}^n
 \left(\int |f-\lambda_k|^2|p_k|^2 \d m\right)^{1/2}\\
 & \leq\sqrt{Q_f(\delta)}+
 \frac{1}{\delta}\left(\frac{1}{n}\sum_{k=1}^n\int |f-\lambda_k|^2|p_k|^2 \d m\right)^{1/2}\\
 & = \sqrt{Q_f(\delta)}+\frac{1}{\delta}\sqrt{\frac{E_n^2}{n}+b_n},
\end{align*}
that gives us the estimate
\begin{equation*}
 \frac{E_n^2}{n}
 \geq\delta^2\left(\Cat_n^{-1/n}
       -\sqrt{Q_f(\delta)}\right)_+^2-{b_n}.
 \label{eq:direct-finite-obstruction}
\end{equation*}
Since $\Cat_n^{-1/n}\to1/4$ and $b_n\to0$, taking the lower limit proves \eqref{eq:concentration-obstruction}.

To conclude, choose $\eta$ such that
\begin{equation*}
    \sup_{q\in\R}m_f(\lbrace q\rbrace)<\eta<\frac{1}{16}.
\end{equation*}
We show that $Q_f(\delta)<\eta$ for some $\delta>0$. If this was not true, $Q_f(1/j)\geq\eta$ for every positive integer $j$ and we may choose $a_j\in\R$ with
\begin{equation*}
    m_f\left([a_j-1/j,a_j+1/j]\right)\geq\eta-1/j.
\end{equation*}
For large $j$, these sets have positive measure. Since $f$ is bounded, the centers $a_j$ are bounded; after passing to a subsequence, $a_j\to a$. For every $\varepsilon>0$, we eventually have
\begin{equation*}
    \lbrace t\in[0,2\pi]\colon|f(t)-a_j|\leq 1/j\rbrace\subseteq\lbrace t\in[0,2\pi]\colon|f(t)-a|\leq\varepsilon\rbrace.
\end{equation*}
It follows that $m_f\left([a-\varepsilon,a+\varepsilon]\right)\geq\eta$. Letting $\varepsilon\downarrow0$ and using continuity from above gives
\begin{equation*}
    m_f\left(\lbrace a\rbrace\right)\geq\eta>\sup_{q\in\R}m_f(\lbrace q\rbrace),
\end{equation*}
and this is a contradiction. Thus $Q_f(\delta)<\eta<1/16$ and we the claim follows.
\end{proof}

\begin{exmp}[The cosine symbol]
Let $f(t)=\cos t$. The substitution $x=\cos t$ and a direct computation show that $f$, seen as a spectral symbol, defines the arcsine distribution (see, for example, \cite{Golinskii-Serra-Jacobi}) and
\begin{equation*}
    Q_{\cos}(\delta)
 =\frac1\pi\int_{1-2\delta}^{1}\frac{\d x}{\sqrt{1-x^2}}
 =\frac{\arccos(1-2\delta)}{\pi}
 =\frac2\pi\arcsin\sqrt{\delta}.
\end{equation*}
In order to force a positive error, we choose $\delta_{\ast}=\sin^2(\pi/50)$, and we obtain
\begin{equation*}
 \liminf_{n\rightarrow\infty}\frac{1}{n}
 \left\|T_n(\cos t)-\PP_{K_n}\left(T_n(\cos t)\right)\right\|_{\fro}^2
 \ge
 \delta_*^2\left(\frac14-\frac15\right)^2
 =\frac{\sin^4(\pi/50)}{400}>0.
\end{equation*}
This provides an explicit obstruction for the cosine.
\end{exmp}

As a consequence, we obtain the following result that completes Corollary~\ref{corosin}.

\begin{coro}
No sequence of matrix algebras $\mathbb{A}(K_n)$ with $K_n$ orthogonal Cauchy-like matrix can ever satisfy the hypotheses \eqref{cond3} of Theorem \ref{teo3}.
\end{coro}

\subsection{Numerical examples}\label{sec6}

We conclude this section by testing our theoretical results. These computations provide visual evidence of the impossibility of satisfying the Korovkin test.
Let $K_n=D_v\cau(x,y)D_w$ be a $n\times n$ orthogonal Cauchy-like matrix. We compute the eigenvalues of the projection $\PP_{K_n}(T_n(f))$, or equivalently $\lambda_k(f(t),n)$ for $k=1,\ldots,n$, where the symbol $f$ can be the Heaviside function (extended periodically) $f=0$ on $(-\pi,0)$, $f=1$ on $(0,\pi)$, and $f=1/2$
at the discontinuities, $f(t)=\cos t$ or $f(t)=\cos t+\cos(2t)/2$, for four different choices of nodes:
\begin{enumerate}[\normalfont i.]
    \item (uniform) $x_k=\frac{k}{n}$, $y_k=\frac{k-1/2}{n}$, vector $w$ with positive entries and signs of the vector $v$ that obeys the rule $[1;1;-1;-1]$,
    \item (uniform-arit.\ prog.) $x_k=\frac{k}{n}$, $y_k=x_k-\frac{1/3+s_k/6}{n}$, vector $w$ with positive entries and signs of the vector $v$ that obeys the rule $[1;-1;1]$,
    \item (uniform-trig.) $x_k=\frac{k}{n}$, $y_k=x_k-\frac{\arcsin\left((1+\sin^{2}(s_k\pi/2))\pi/16\right)}{n\pi}$, vector $w$ with positive entries and signs of the vector $v$ that obeys the rule $[1;1;-1;-1]$,
    \item (trig.) $x_k=\cos\left(\frac{(n-k+1/2)\pi}{n}\right)$, $y_k=\cos\left(\frac{(n-k+1)\pi}{n}\right)$ and vectors $w$ and $v$ with positive entries.
\end{enumerate}
In cases ii.\ and iii.,
$s_k=(k-1)/n$, for $k=1,\ldots,n$. Case iv.\ is also investigated in \cite[Section 5]{fasino-cauchy}. We set the dimension to $n=600$. The evaluations of the Toeplitz symbols and the eigenvalues of the projection are normalized and rearranged in monotone non-decreasing order in Figure~\ref{fig}. As it can be observed, the eigenvalues fail to approximate the corresponding symbols, both in the case of smooth generating functions and in the case of the locally constant Heaviside function. This is in line with the Korovkin theory for which a classical dichotomy holds: either the approximation stands for any $f$ at least continuous and $2\pi$ periodic or it fails practically always, with few exceptions including the constant function.

\begin{figure*}[!ht]
\centering
\includegraphics[width=.28\textwidth]{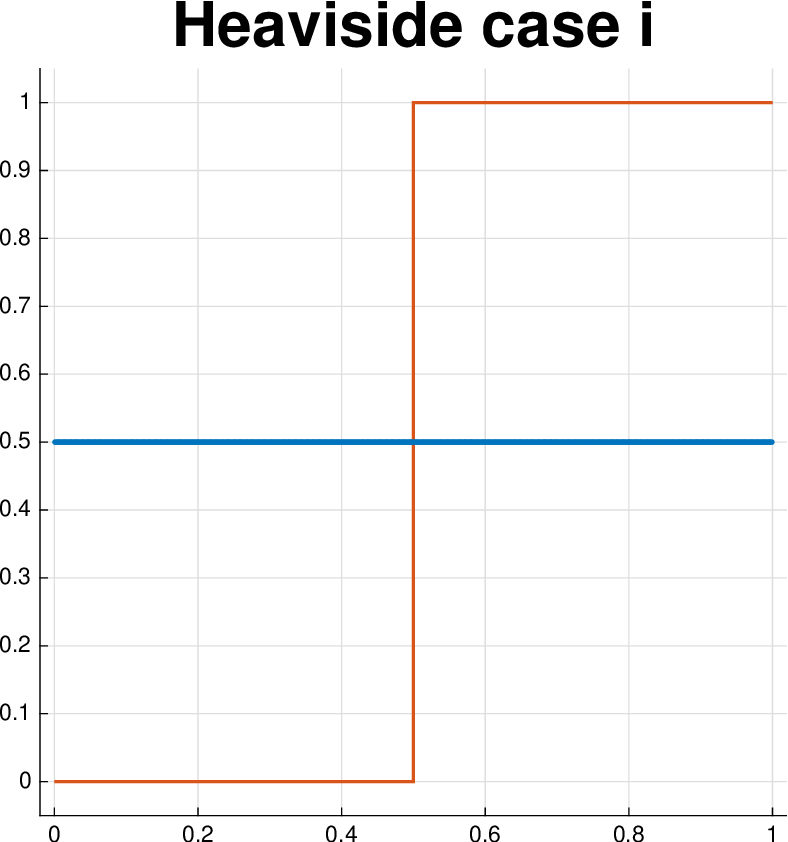}
\includegraphics[width=.28\textwidth]{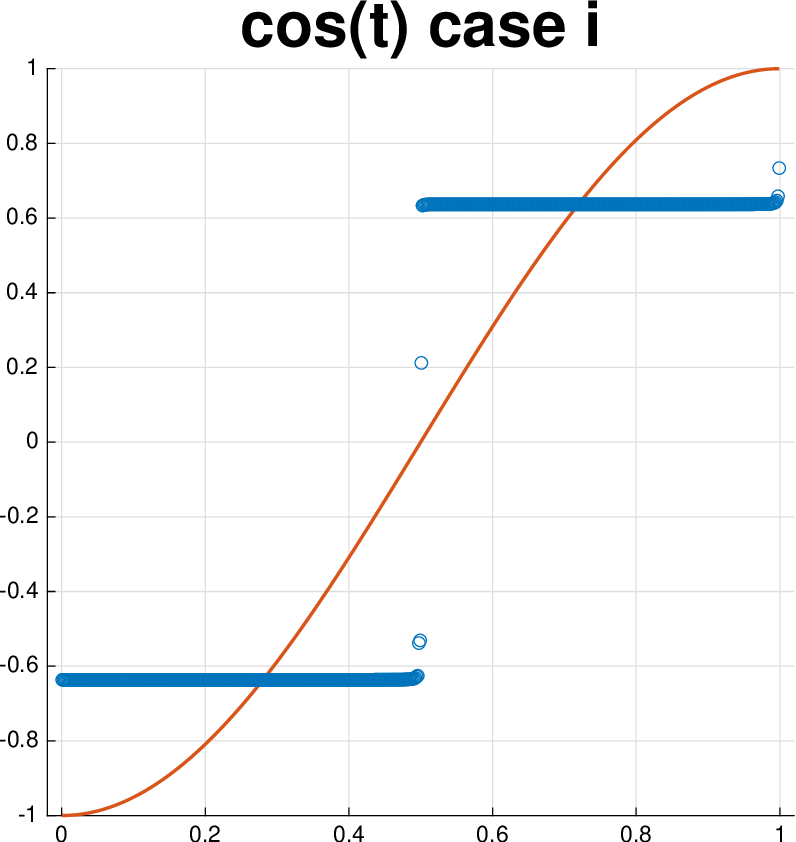}
\includegraphics[width=.28\textwidth]{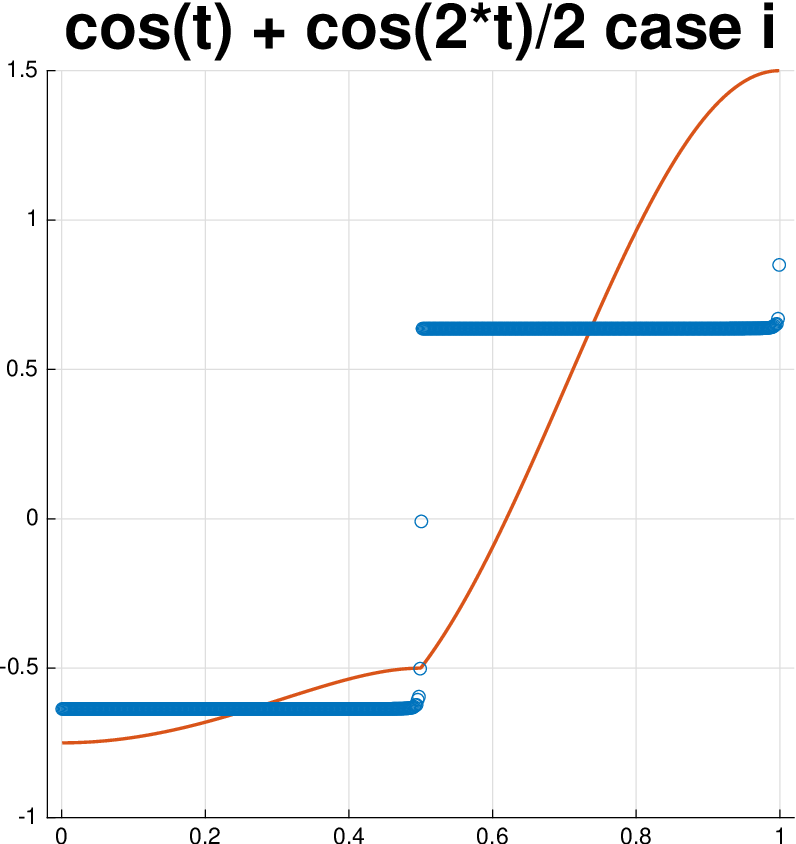}\\
\includegraphics[width=.28\textwidth]{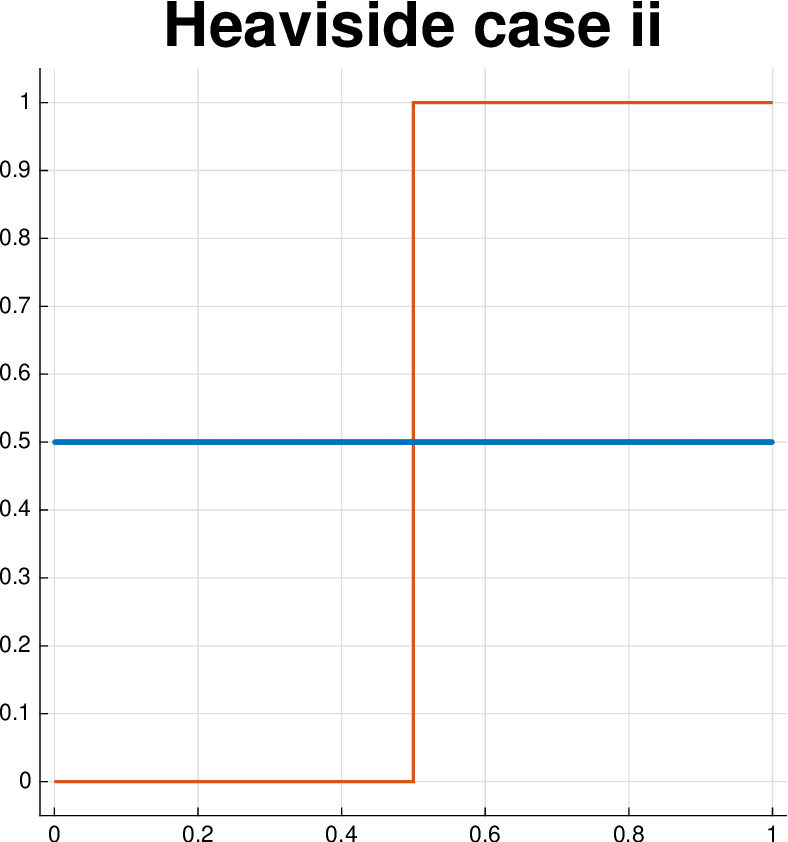}
\includegraphics[width=.28\textwidth]{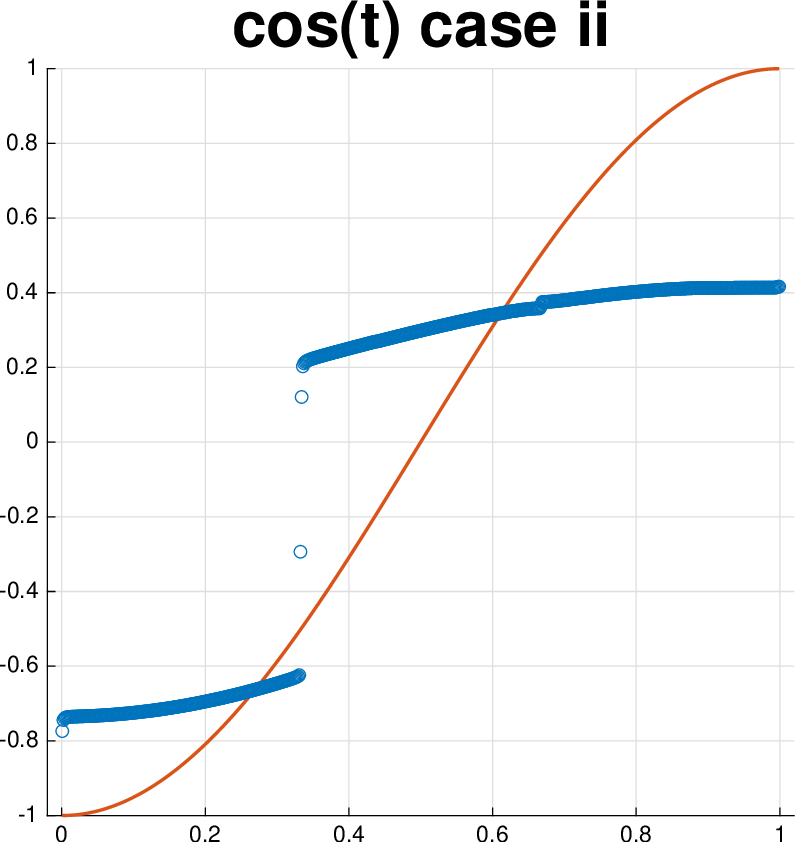}
\includegraphics[width=.28\textwidth]{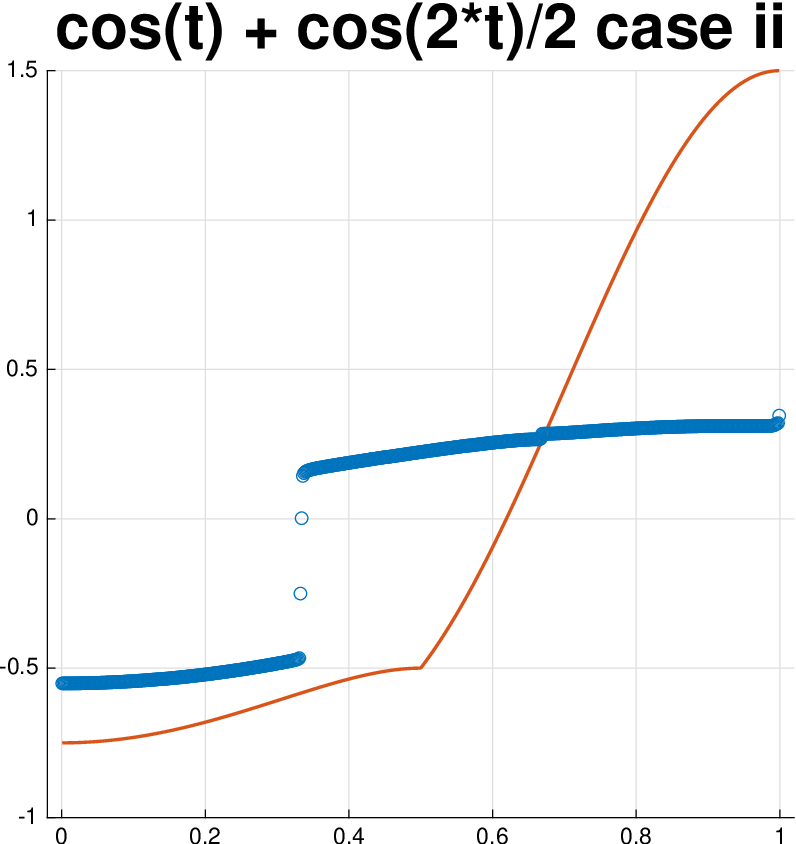}\\
\includegraphics[width=.28\textwidth]{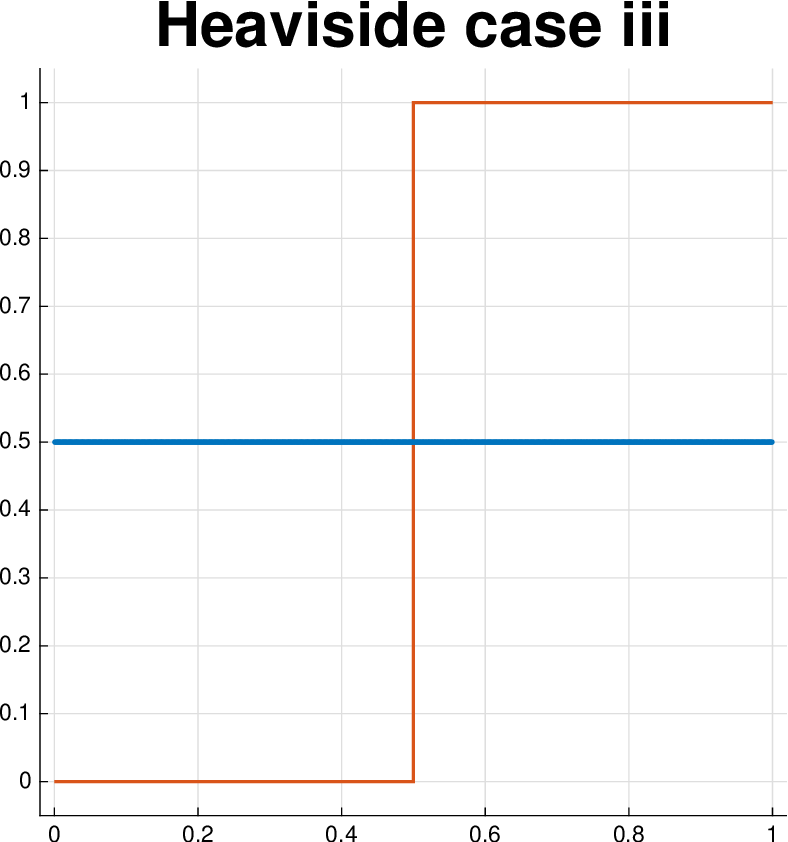}
\includegraphics[width=.28\textwidth]{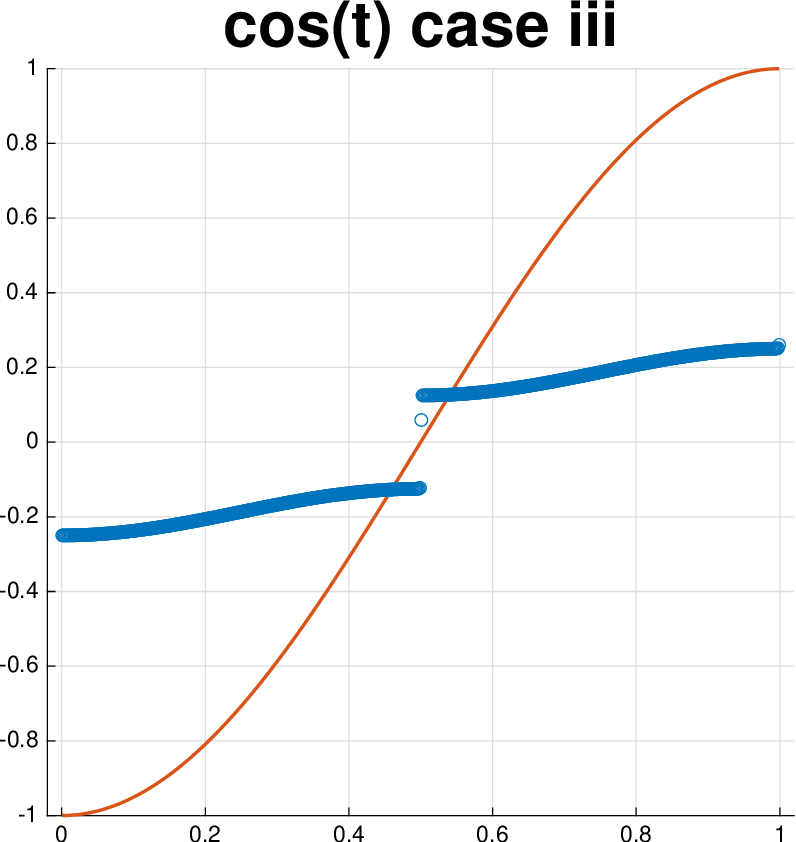}
\includegraphics[width=.28\textwidth]{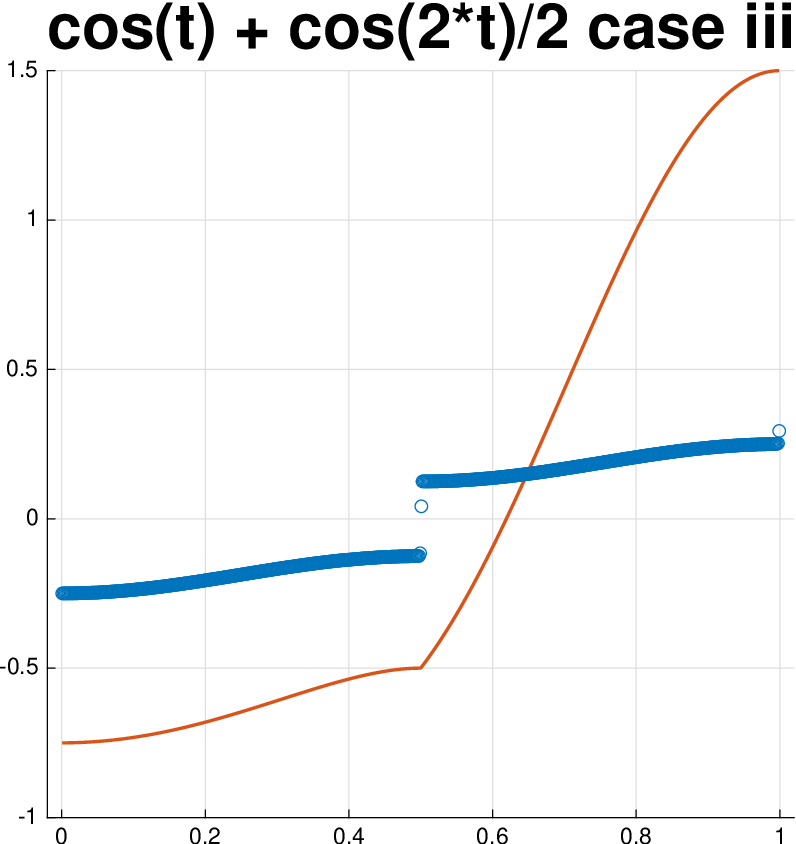}\\
\includegraphics[width=.28\textwidth]{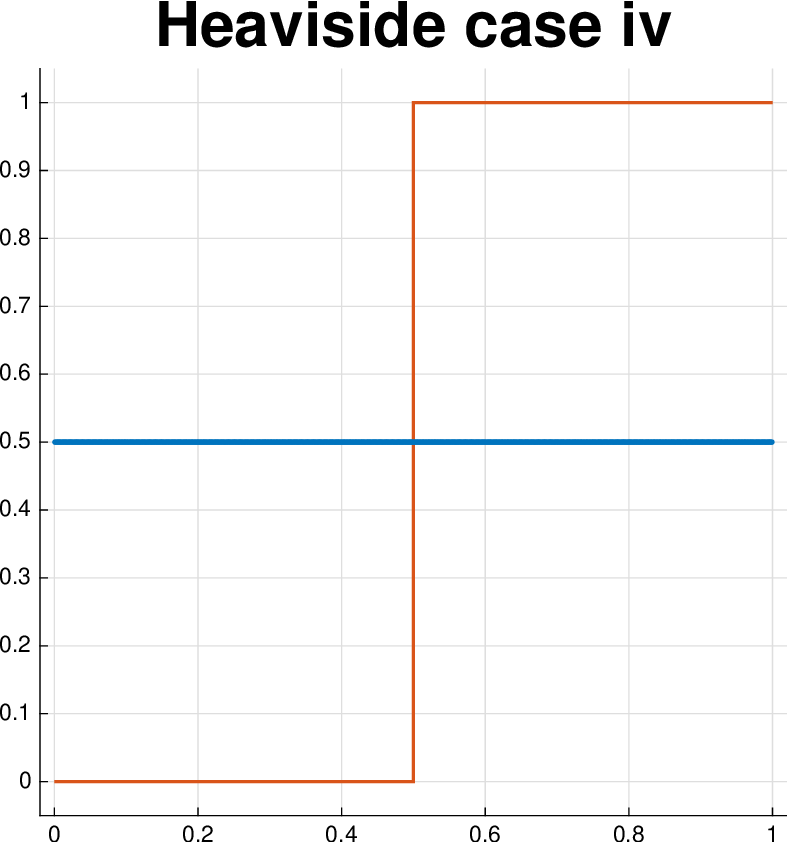}
\includegraphics[width=.28\textwidth]{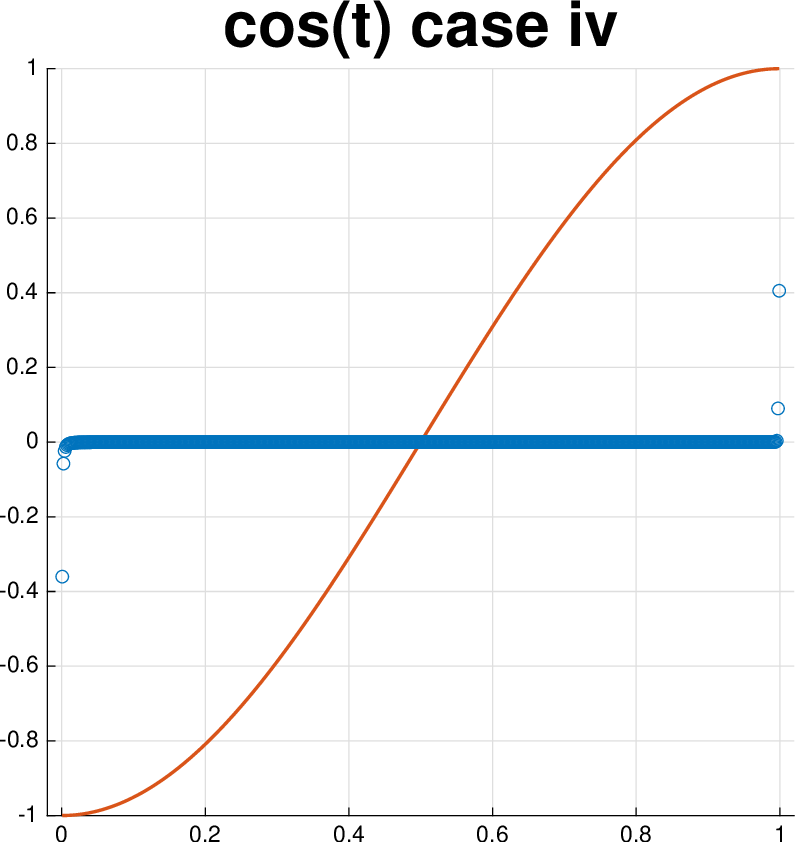}
\includegraphics[width=.28\textwidth]{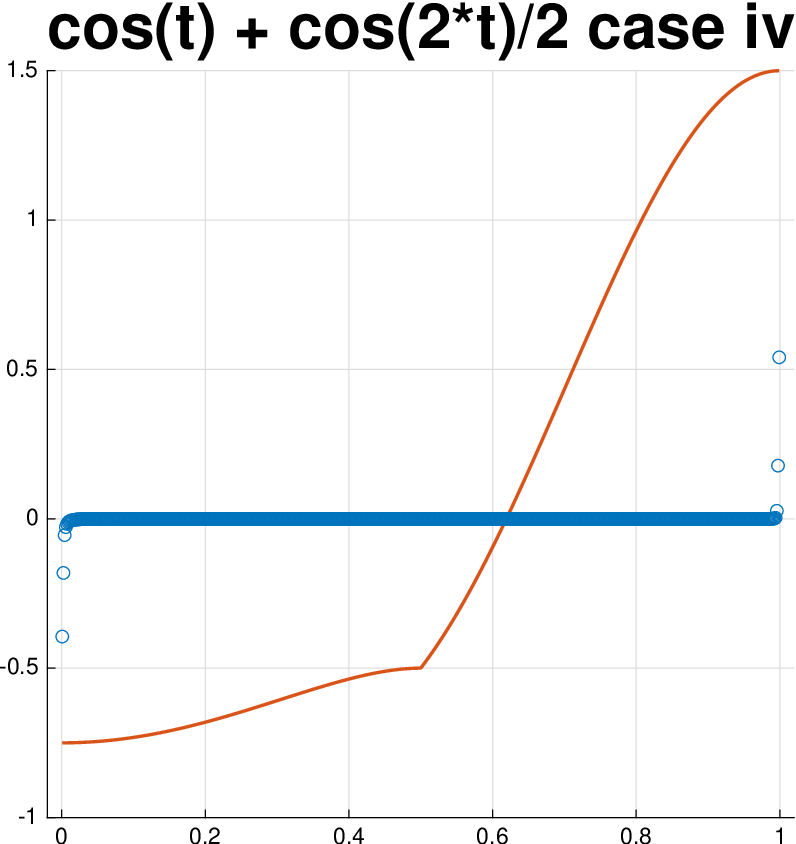}
\caption{Monotone rearrangement comparison of the symbols and the projected eigenvalues normalized on the interval $[0,1]$ for the four choices of nodes.}\label{fig}
\end{figure*}

\section{Conclusion}\label{sec:end}

In the present work we have considered newly designed matrix algebras associated with Cauchy orthogonal transforms. We tested the approximation of Toeplitz matrices
by these algebras, in the spirit of finding optimal methods based on preconditioning. We used the Korovkin theory to check such approximability but, unfortunately, the Korovkin test turned out to be negative.
This is quite surprising given the fact that the Korovkin test was proven to be successful also in the block and multilevel Toeplitz setting, for all the known trigonometric, Fourier-like, and Hartley algebras \cite{Koro-algebra1,koro1,Koro-block,Koro-4-fourier}. Hence, for a given Cauchy algebra and for a matrix size $n$ large enough, it is true that the linear space of the Toeplitz matrices is almost orthogonal to the approximation space constituted by the considered Cauchy algebra. It is also worth noticing that other negative approximation results in the multilevel Toeplitz setting can be found in \cite{nega-gen,SeTy2}.
Still a few questions remain open and among them we mention the following:
\begin{itemize}
  \item Can the algebra generated by a product of a certain number of Cauchy-like matrices approximates the space of Toeplitz matrix sequences with $L^1$ generating functions? Can this be obtained by considering matrices with higher displacement rank?
  \item Given a given Cauchy algebra, is there a way to understand a priori what kind of structured matrices can be approximated by that algebra?
  \item Can successful approximation results be obtained by exploiting complex unitary Cauchy-like matrices?
\end{itemize}
In conclusion we think that the present contribution open the door to a quite rich path of future investigations.

\backmatter

\bmhead{Acknowledgements}

The work of all the authors is supported by the Italian Agency INdAM-GNCS.
The work of Stefano Serra-Capizzano is funded from the European High-Performance Computing Joint Undertaking  (JU) under grant agreement No 955701. The JU receives support from the European Union’s Horizon 2020 research and innovation programme and Belgium, France, Germany, Switzerland.
Furthermore Stefano Serra-Capizzano is grateful for the support of the Laboratory of Theory, Economics and Systems – Department of Computer Science at Athens University of Economics and Business.

%
%
%
%


\end{document}